\documentclass[12pt, reqno]{amsart}
\usepackage{graphicx, amssymb, amsmath, amsthm}
\numberwithin{equation}{section}
\usepackage{microtype}

\usepackage{xspace}

\usepackage[]{todonotes}		
\usepackage{mathtools}

\usepackage{vmargin}

\usepackage{xcolor}

\usepackage[colorlinks = true,
			citecolor = blue,
			linkcolor = blue]{hyperref}

\let\pa\partial

\newcommand{\dd}{\mathop{}\!\mathrm{d}}
\newcommand{\R}{\mathbb{R}}

\newcommand{\N}{\mathbb{N}}
\newcommand{\AH}{H_{\textup{par}}}
\newcommand{\CZ}{Calder\'on--Zygmund}
\newcommand{\FDB}{Fa\`a di Bruno formula}

\newtheorem{theorem}{Theorem}[section]

\newtheorem{lemma}[theorem]{Lemma}

\newtheorem{corollary}[theorem]{Corollary}

\theoremstyle{definition}

\newtheorem{remark}[theorem]{Remark}

\begin{document}
\title[A Mikhlin--H\"ormander multiplier theorem]{ A Mikhlin--H\"ormander multiplier theorem for the partial harmonic oscillator}

\author[X. Su]{Xiaoyan Su}
\address{Laboratory of Mathematics and Complex Systems (Ministry of Education) \\ School of Mathematical Sciences\\
 Beijing Normal University, Beijing 100875,  China}
\email{suxiaoyan0427@qq.com}

\author[Y. Wang]{Ying Wang}
\address{Graduate School of China Academy of Engineering Physics,  \ Beijing, \ China, \ 100088}
\email{wsming@bupt.cn}

\author[G. Xu]{Guixiang Xu}
\address{Laboratory of Mathematics and Complex Systems,\
Ministry of Education,\
School of Mathematical Sciences,\
Beijing Normal University,\
Beijing, 100875, People's Republic of China.
}
\email{guixiang@bnu.edu.cn}

\subjclass[2010]{42B15,  42B25.}

\keywords{Littlewood--Paley $g$-function; Mikhlin--H\"ormander multiplier;  Partial harmonic oscillator; The Mehler formula. }

\begin{abstract}
We prove a Mikhlin--H\"ormander multiplier theorem for the partial harmonic oscillator $H_{\textup{par}}=-\pa_\rho^2-\Delta_x+|x|^2$ for $(\rho, x)\in\R\times\R^d$ by using the Littlewood--Paley  $g$ and $g^\ast$ functions and the associated heat kernel estimate. The multiplier we have investigated is defined on $\mathbb R \times \mathbb N$.

\end{abstract}

 \maketitle


\section{Introduction}
\noindent



In this paper, we prove a  Mikhlin--H\"ormander multiplier theorem for the partial harmonic oscillator in $\R^{d+1}$:
\begin{equation*}
  \AH=-\partial_{\rho}^2-\partial_{x_1}^2-\dots-\partial_{x_d}^2+|x|^2.
\end{equation*}
The Schr\"odinger flows for the
 operator $\AH$ arises in various branches of physics, such
as the Bose--Einstein condensates, and the propagation of mutually incoherent wave
packets in nonlinear optics (see \cite{JP}). 
%

The classical Mikhlin--H\"ormander multiplier theorem states  that  for $1<p<\infty$,
$
    \|(m\widehat f)^\vee \|_p \leq C_{p, d} \|f\|_p
$
provided that the Fourier multiplier $m \in C^{\lfloor \frac{d}{2}\rfloor +1}(\mathbb R^{d}\setminus \{0\})$ and satisfies
$
    |\partial^\alpha m(\xi)| \leq C_{\alpha} |\xi|^{-\alpha}
$ for all multi-indices $\alpha$ with $|\alpha| \leq \lfloor \frac{d}{2}\rfloor +1$.
This result can be proved  either by the {\CZ}  singular integral operator theory in  \cite{Grafakos249,MuscaluSchlag}, or
by the Littlewood--Paley $g$-functions in  \cite{Stein}. The use of the Fourier transform stems from the fact that the Laplacian operator only has the continuous spectrum in  $\mathbb R^d$.



For the operators with discrete spectrum, such as the spherical Laplacian operator $-\Delta_{\mathbb S^{d}}$, or the Hermite operator $-\Delta_x + |x|^2$, a sufficient condition to guarantee the $L^p$-boundedness of multipliers is the proper decay in the finite differences. More precisely, the multiplier operator for Hermite expansions is defined by
\begin{align*}
    T_m f(x) =\sum_{\mu\in\N^d} m(2|\mu|+d)( f(x), \Phi_{\mu}) \Phi_{\mu}(x),
\end{align*}
where $\Phi_\mu$ is a Hermite function, see Section \ref{subseq:hermite-fns} below.
By Theorem 1 in \cite{Than87}, $T_m$ is bounded on $L^p(\mathbb R^d)$ for $1<p<\infty$ provided that
\begin{align*}
    |\triangle_k ^j m(k)| \leq C_N k^{-j}\  \text{for} \ j=0, 1, \dots, N,
\end{align*}
whenever $N>\frac{d}{2}$, where $\triangle_k^j$ is the $j$-th forward finite difference.
This above result is shown by use of the Littlewood--Paley $g$-functions, see \cite{Bonami, Than87}.

In this paper,  our goal is  to show a Mikhlin--H\"ormander multiplier theorem for the Schr\"odinger operator $\AH$, which serves as an example for which the multiplier is defined in both continuous  and discrete variables.
We remark that  the operator $\AH$ is a polynomial perturbation of the Laplacian operator. Some multiplier results and Littlewood--Paley square function estimates for operators with polynomial perturbations have been established in \cite{Dziubanski98, Dziubanski97, Dziubanski99} by using nilpotent Lie algebras.  Recently, Killip, Miao, etc,  make use of by the {\CZ}  singular integral operator theory in  \cite{Grafakos249,MuscaluSchlag} to show the Mikhlin-H\"ormander multiplier theorem for the Schr\"odinger operator $\mathcal{L}_a:=-\Delta+\tfrac{a}{|x|^2}$,  $a\geq-\tfrac{(d-2)^2}4$ in \cite{KMVZZ}. This result was crucially used in \cite{KMVZZ2017, MMZ} to obtain the scattering result of the solution for nonlinear Schr\"odinger and wave equations with the inverse-square potential. 

Our method closely relies on the structure of the operator $\AH$ and the Mehler formula, and offers a different view towards understanding the operator $\AH$. We can refer to a companion paper \cite{SuWangXu} for the Riesz transform and Sobolev spaces associated to  the operator $\AH$.

\subsection{Main result}
For smooth function $f\in C_0^\infty(\mathbb R^{d+1})$, $\AH f$ can be reformulated by Fourier analysis as follows:
\begin{align}\label{H decomposition}
     \AH f(\rho,x)&=
    \sum_{\mu\in\N^d} \frac{1}{\sqrt{2\pi}}\int_{\mathbb R} e^{i\tau \rho}(\tau^2+2|\mu|+d)  (\mathcal F_{\rho} f(\tau, \cdot), \Phi_{\mu}(\cdot)) \Phi_{\mu} (x)\dd \tau \notag\\
   &= \sum_{k=0}^\infty\frac{1}{\sqrt{2\pi}} \int_{\mathbb R} e^{i\tau \rho}(\tau^2+2k+d)P_k \mathcal F_{\rho} f(\tau, x) \dd \tau,
    \end{align}
where $\mathcal F_{\rho} f$ is the Fourier transform with respect to $\rho$,  and $P_k$ is the projection to the $k$th eigenspace of the operator $\AH$ in $x$, which is spanned by the eigenfunctions $\Phi_\mu$'s for $|\mu|=k$: see Section \ref{e:P_k} below.

Let  $m=m(\tau,k)$ be defined on $\mathbb R\times \mathbb N$. We define the operator $T_m$ for $\AH$ by
\begin{align}\label{equ:operm}
  T_mf(\rho, x)&=\sum_{k=0}^{\infty}\int_{\R}e^{i\tau\rho}
m(\tau,k)P_k(\mathcal{F}_\rho f)(\tau,x)\dd \tau, \quad \text{for any }\; f \in C_0^\infty(\mathbb R^{d+1}).
\end{align}
In particular, if  $m(\tau,k)=m(\tau^2+2k+d)$,  the  multiplier operator $T_m$ coincides with $m(\AH)$ defined by the functional calculus (see \cite{SuWangXu}),  so the multipliers defined in \eqref{equ:operm} are more general than those defined by the spectral measure.

 The main result in this paper is as follows:
\begin{theorem}[Mikhlin--H\"ormander multiplier]\label{thm:MikhlinmulHP}
Suppose that  a function $m(\tau,k)$ defined on $\mathbb R \times \mathbb N$ satisfies the estimates
\begin{gather}\label{equ:assummulf}
 \Big| \frac{\pa^N}{\pa\tau^N}m(\tau,k)\Big|
 \leq C(\tau^2+2k+d)^{-\frac{N}{2}}  \ \text{and} \  \big|\triangle_k^N
 m(\tau,k)\big|\leq C(\tau^2+2k+d)^{-N}
\end{gather}
for all $0\leq N\leq \lfloor \frac{d+1}2\rfloor+1$. Then,  we have for any $1<p<\infty$
\begin{align*}
   \|T_m f\|_{L^p(\mathbb R^{d+1})}  \leq C \|f\|_{L^p(\mathbb R^{d+1})}.
\end{align*}
\end{theorem}

\begin{remark}
The similar result for the generalized
partial harmonic oscillator $-\Delta_y-\Delta_x+|x|^2$ with $y\in\R^{d_1}$ and $x\in\R^{d_2}$ holds by the same argument.
\end{remark}

  Let  $m(r)\in C_0^\infty(\mathbb R;[0,1])$  with $\operatorname{supp} m \subseteq [\frac38,\frac 34]$. Denote $m_j(r)=m(2^{-j} r)$, and the operator $ \Delta_j f(\rho, x)=T_{m_j(\sqrt{\tau^2+2k+d})}  f(\rho, x).$  As a direct consequence of Theorem \ref{thm:MikhlinmulHP} and Khintchine's inequality, we have the following Littlewood--Paley square function estimates for the operator $\AH$.
\begin{corollary}
For $1<p<\infty$, we have
 \begin{align*}
 \|f\|_{L^p(\mathbb R^{d+1})} \simeq \bigg\|\bigg(\sum_{j=0}^\infty |\Delta_j f|^2\bigg)^{1/2}\bigg\|_{L^p(\mathbb R^{d+1})}.
 \end{align*} Furthermore, for $\alpha\ge 0$, $1<p<\infty $, the Sobolev spaces ${W}_{\AH}^{\alpha, p}$ associated to the operator $\AH$ (see \cite{SuWangXu}) can be characterized by
 \begin{align*}
    \|f\|_{{W}_{\AH}^{\alpha, p}(\mathbb   R^{d+1})} \simeq \Big\|\big(\sum_{j=0}^\infty |2^{j\alpha}\Delta_j f|^2\big)^{1/2}\Big\|_{L^p(\mathbb R^{d+1})}.
 \end{align*}
\end{corollary}
We omit the proof and the readers can refer  to \cite{Grafakos249, MuscaluSchlag}.

Lastly, this paper is organized as follows:  in Section \ref{subseq:hermite-fns}, we introduce some preliminary results about Hermite functions, the Mehler formula and the heat kernel of the operator $\AH$.   In Section \ref{sect:proof}, we  show the proof of Theorem \ref{thm:MikhlinmulHP}.

\subsection*{Acknowledgements} The authors would like to thank Professor Changxing Miao for his valuable comments and suggestions. G. Xu  was supported by National Key Research and Development Program of China (No. 2020YFA0712900) and by NSFC (No. 11831004).

\section{Preliminaries}\label{subseq:hermite-fns}
\subsection{Hermite functions}
We first recall the Hermite functions on $\R^d$ as in \cite{Than93book}.
  The Hermite functions $h_k$ on $\mathbb R$ are defined by
 \begin{align*}
     h_k(x)=(2^k k! \sqrt{\pi})^{-1/2} (-1)^k \frac{\dd^k}{\dd x^k}(e^{-x^2}) e^{-x^2/2}.
 \end{align*}
 Let $\mu=(\mu_1,\dots,\mu_d)$ be a multi-index and $x\in\R^d$. The Hermite functions $\Phi_\mu$ on $\mathbb R^d$, is defined by taking the product of the 1-dimensional Hermite functions $h_{\mu_j}(x_j)$:
\begin{equation*}
  \Phi_\mu(x)=\prod_{j=1}^{d}h_{\mu_j}(x_j).
\end{equation*}
The functions $\Phi_\mu$ form a complete orthonormal system for $L^2(\R^d)$.
If we define  the operators $A_j=-\frac{\pa}{\pa x_j}+x_j$ for $1\leq j \leq d$, then
\begin{align}\label{equ:ajphiest}
A_j \Phi_{\mu}  =\sqrt{2(\mu_j+1)}  \Phi_{\mu+e_j},
\end{align}
where $e_j$ is the $j$th coordinate vector in $\mathbb N^d$.

Denote by $P_k$ the spectral projection to the $k$th eigenspace of $-\Delta_x+|x|^2$,
 \begin{align}
P_kf(x)=\int_{\mathbb R^d} \sum_{|\mu|=k}\Phi_\mu(x)\Phi_\mu(x') f(x')\dd x'. \label{e:P_k}
 \end{align}
 These projections  are the integral operators with kernels
 \begin{align*}
     \Phi_k(x,x')=\sum_{|\mu|=k}\Phi_\mu(x)\Phi_\mu(x').
 \end{align*}
The Mehler formula for $ \Phi_k(x,x')$ is
\begin{align}\label{Mehler's formula}
   \sum_{k=0} ^\infty r^k \Phi_k(x,x') =\pi ^{-d/2}(1-r^2)^{-d/2} e^{-\frac{1}{2}\frac{1+r^2}{1-r^2}(|x|^2+|x'|^2)+\frac{2r x \cdot x'}{1-r^2}},
\end{align}
for $0<r<1$, see \cite[p. 6]{Than93book}.

The following lemmas are the direct consequences of the Mehler formula,  which will be frequently used in the next section.

\begin{lemma}[\cite{Than93book}, P. 92]  For all $t>0$, we have
%
  \begin{gather}\label{equ:phisumest}
  \sum_{\mu\in\N^d}e^{-t|\mu|}\Phi_{\mu}(x)^2\lesssim t^{-\frac{d}{2}},\;\forall\;x\in\R^d,
  \\
%
%
  \label{equ:intprop}
  \int_{\R^d}\Big(\sum_{\mu\in\N^d}e^{-t(2|\mu|+d)}\Phi_\mu(x)^2\Big)\dd x
  =C(\sinh t)^{-d}.
\end{gather}
\end{lemma}

\subsection{Heat kernel for the operator $\AH$}  We write $z=(\rho, x)$ or $z'=(\rho', x')$ to denote variables in $\mathbb R^{d+1}$ with $\rho, \rho' \in \mathbb R$ and $x, x' \in \mathbb R^d$.


From \eqref{H decomposition},  we can define the heat semigroup with $ f\in C_0^\infty(\mathbb R^{d+1})$ as follows:
\begin{align}\nonumber
    e^{-t \AH}f(\rho,x)&=  \sum_{\mu\in\N^d}\frac{1}{\sqrt{2\pi}} \int_{\mathbb R}e^{i\tau \rho} e^{-t(\tau^2+2|\mu|+d)  }  (\mathcal F_{\rho} f(\tau, \cdot), \Phi_{\mu}(\cdot)) \Phi_{\mu} (x)\dd \tau \\ \notag 
    &=\sum_{k=0}^\infty \frac{1}{\sqrt{2\pi}}\int_{\mathbb R} e^{i\tau \rho} e^{-t(\tau^2+2k+d)  }  {P_k}(\mathcal F_{\rho} f)(\tau, x)  \dd \tau\\\nonumber
    &=\int_{\mathbb R^{d+1}} K(t,z,z') f(z') \dd z',
\end{align}
where  we use the Mehler formula \eqref{Mehler's formula} in the last step  and 
\begin{align}\label{kernel of semigroup}
   K(t,z,z')
    =2^{-\frac{d+2}{2}}\pi^{-\frac{d+1}{2}}t^{-1/2}(\sinh 2t)^{-d/2} e^{-B(t,z,z')},
\end{align}
and
\begin{align*}
    B(t, z, z')=\frac{1}{4}(2\coth2t-\tanh t)|x-x'|^2+\frac{\tanh t}{4} |x+x'|^2+\frac{(\rho-\rho')^2}{4t}.
\end{align*}

%



\section{Proof of Theorem \ref{thm:MikhlinmulHP}}\label{sect:proof}
We follow  the  arguments as in \cite{Grafakos249,Stein}. It suffices to show the following estimates
\begin{equation}\label{ests}
  \|T_mf\|_{L^p}\leq C\|g_{N+1}(T_mf)\|_{L^p}\leq C\|g_{N}^\ast(f)\|_{L^p}\leq C\|f\|_{L^p}
\end{equation}
for some integer $N\in\mathbb{N}$.

First, given $N\in\N$, we define the Littlewood--Paley $g_N$-function by
\begin{equation*}
  g_N(f)(z)=\Big(\int_0^\infty\big|\pa_t^N e^{-t\AH}f(z)\big|^2t^{2N-1}\dd t\Big)^\frac12.
\end{equation*}

\begin{lemma}[]\label{prop:l2gn}
For each $N\geq1$ and $f\in L^2(\R^{d+1})$, there holds
\begin{equation}\label{equ:gnfl2equiv}
  \|g_N(f)\|_{L^2(\R^{d+1})}^2=2^{-2N}\Gamma(2N)\|f\|_{L^2(\R^{d+1})}^2.
\end{equation}

\end{lemma}

\begin{proof}
By the orthogonality of Hermite functions, we have
\begin{align*}
    \big\|\pa_t^Ne^{-t\AH}f(\rho,\cdot)\big\|_{L_x^2(\R^d)}^2=\sum_{k=0}^{\infty}\Big\|\int_{\R}e^{i\tau\rho}(\tau^2+2k+d)^N
e^{-t(\tau^2+2k+d)}P_k(\mathcal{F}_\rho f)(\tau,\cdot)\dd \tau\Big\|_{L_x^2(\R^d)}^2.
\end{align*}
It follows from the Plancherel theorem in  $\rho$ that
\begin{align*}
 & \|g_N(f)(z)\|_{L^2(\R^{d+1})}^2=\int_{\R^{d+1}}\int_0^\infty\big|\pa_t^N e^{-t\AH}f(z)\big|^2t^{2N-1}\dd t\dd z\\
 &=\sum_{k=0}^{\infty}\int_{\R^{d+1}}\int_0^\infty\Big|(\tau^2+2k+d)^N
e^{-t(\tau^2+2k+d)}P_k(\mathcal{F}_\rho f)(\tau,\cdot)\Big|^2t^{2N-1}\dd t\dd \tau\dd x\\
 &=\sum_{k=0}^{\infty}\int_{\R^{d+1}}|P_k(\mathcal{F}_\rho f)(\tau,\cdot)|^2\Big[\int_0^\infty(\tau^2+2k+d)^{2N}e^{-2t(\tau^2+2k+d)}t^{2N-1}\dd t\Big]\dd \tau\dd x\\
 &=2^{-2N}\Gamma(2N)\sum_{k=0}^{\infty}\int_{\R^{d+1}}|P_k(\mathcal{F}_\rho f)(\tau,\cdot)|^2\dd \tau\dd x \\ & =2^{-2N}\Gamma(2N)\|f\|_{L^2(\R^{d+1})}^2,
\end{align*}
which completes the proof. \qedhere
\end{proof}

\begin{lemma}[Equivalence of $L^p$ norms]\label{prop:equivnormgn}
Let $1<p<\infty$ and $N\in\N$. Then, there exist  $C_{1},C_{2}>0$ such that for all $f\in L^p$, we have
\begin{equation*}
 C_{1}\|f\|_{L^{p}(\R^{d+1})}\leq \|g_N(f)\|_{L^p(\R^{d+1})}\leq C_{2}\|f\|_{L^p(\R^{d+1})}.
\end{equation*}
\end{lemma}
\begin{proof}
The fact that $\|g_N(f)\|_{L^2(\mathbb R^{d+1})} = C\|f\|_{L^2(\mathbb R^{d+1})}$ is given by  Lemma \ref{prop:l2gn}.  For general case $p\in(1,\infty)$, we will view
$g_N$ as a singular integral with a kernel taking values in the Hilbert space $\mathcal{H}_2:=L^2(\R^+;t^{2N-1}dt)$ by the auxillary function
\begin{equation*}
  \tilde{g}_N(f)(t,z)=
  \int_{\R^{d+1}}\frac{\pa^N K(t, z, z') }{\pa t^N} f(z')\dd  z',
\end{equation*}
where $K(t, z, z')$ is the kernel  \eqref{kernel of semigroup}.
By definition, we have
\begin{align*}
  \|\tilde{g}_N(f)(\cdot,z)\|_{\mathcal{H}_2}^2
 & =g_N(f)(z)^2,\\
\big\|\|\tilde{g}_N(f)(\cdot, z)\|_{\mathcal{H}_2}  \big\|_{L^p(\R^{d+1})} &=  \|g_N(f)\|_{L^p(\R^{d+1})}.
\end{align*}
The kernel of $\tilde{g}_N(f)$ is
$$G_N(t,z, z')=\frac{\pa^N K(t, z, z')}{\pa t^N} . $$
We claim the following facts hold:
\begin{gather}
 |G_N(t,z, z')|\lesssim t^{-\frac{d+1}{2}-N} e^{-\frac{1}{16t}|z-z'|^2}, \label{claim 1}\\
    |\partial_z G_N(t,z, z')|+ |\partial_{z'} G_N(t,z, z')|\lesssim t^{-\frac{d+2}{2}-N} e^{-\frac{1}{16t}|z-z'|^2}\label{claim 2}.
\end{gather}

Now we show the estimates \eqref{claim 1} and \eqref{claim 2}.  For $N=0$, the basic estimates that
 \begin{align*}
    2\coth2t-\tanh t >\coth 2t>&\tfrac{1}{2t},\\
    \tanh t>t,\;
    \sinh 2t\geq&  t
 \end{align*}
imply the upper bound
\begin{align*}
   | K(t, z, z')|\leq C t^{-\frac{d+1}{2}} e^{-\frac{1}{8 t}|z-z'|^2} e^{-\frac{t}{4 }|x+x'|^2} .
\end{align*}
Let $N\geq 1$.
By using the high order derivative formula
\begin{align*}
    \frac{\dd ^N \sinh t }{\dd t^N} &=-i^{N+1} \sin \Big(it +\frac{\pi N}{2}\Big)
\end{align*}
and the \FDB{}, we get
\begin{align*}
    \frac{\dd ^N (\sinh t)^{-d/2}}{\dd t^N}=\sum C_{N, m_1, \dots m_N} (\sinh t)^{-d/2-(m_1+\dots +m_N)} \prod_{j=1}^N \Big(\frac{\dd ^{j} \sinh t }{\dd t^{j}}\Big)^{m_j},
\end{align*}
where the sum is over all $m_i\in\mathbb Z_{\ge 0}$ such that $m_1+2m_2+\dots+ Nm_N=N$. As
\begin{align*}
   \Big|\sin \Big(it +\frac{\pi N}{2}\Big)\Big|
   \lesssim
   \begin{cases}
 1, & 0<t<1,\\
 e^t, & t>1,
   \end{cases}
\end{align*}
it follows that
\begin{align}\label{bound for derivatives of sinh}
   \left| \frac{\dd ^N (\sinh 2t)^{-d/2}}{\dd t^N}\right|\lesssim t^{-d/2-N}.
\end{align}
To estimate the derivatives for $B(t,z, z')$, we use the following  formulas:
\begin{align*}
  \frac{\dd ^N \coth t }{\dd t^N} &=(-1)^{N} 2^{N+1} \operatorname{Li}_{-N}(e^{-2t}), \\
 \frac{\dd ^N \tanh t }{\dd t^N} &= - 2^{N+1} \operatorname{Li}_{-N}(-e^{2t}),
\end{align*}
where $\operatorname{Li}_{-N}$ is the polylogarithm in \cite{Lewin1981}.
Hence, we have
\begin{align}\label{derivative of B}
   \frac{\partial^N B(t, z, z')}{\partial t^N}= & \frac{1}{4}\Big[2^{N+1}(-1)^{N} 2^{N+1} \operatorname{Li}_{-N}(e^{-4t})- 2^{N+1} \operatorname{Li}_{-N}(-e^{2t})\Big] |x-x'|^2 \nonumber\\
   &- 2^{N-1} \operatorname{Li}_{-N}(-e^{2t})
 |x+x'|^2+C_N\frac{(\rho-\rho')^2}{t^{N+1}}.
\end{align}
Since
$    \lvert\operatorname{Li}_{-N}(s)|\lesssim 1 $ for $0<s<1/2,$
we have
$
   \lvert\operatorname{Li}_{-N}(e^{-t})|\lesssim 1\ \text{when }  t>1.
$
By $(7.187)$ and $(7.191)$ in \cite{Lewin1981},  we also have
$
    \lvert\operatorname{Li}_{-N}(-e^{t})|\lesssim 1\ \text{when }   t>1.
$

When $0<t<1$, by the Laurent expansions of $\tanh t$ and $\coth t$ , we have
 \begin{align*}
     \Big| \frac{\dd ^N \coth t }{\dd t^N}\Big|\lesssim  t^{-(N+1)}, \quad  \Big| \frac{\dd ^N \tanh t }{\dd t^N}\Big|\lesssim  1.
 \end{align*}
From this and \eqref{derivative of B}, we have for $N\geq 1$ that 
\begin{align}\label{bound for derivatives of B}
 \Big | \frac{\partial^N B(t, z, z')}{\partial t^N}  \Big | \lesssim t^{-(N+1)} |z-z'|^2+  |x+x'| ^2.
\end{align}
Direct computation gives  the following upper bound for the derivatives in  $z$ and $z'$,
\begin{align}\label{derivative in z}
 \Big | \frac{\partial^{N+1} B(t, z, z')}{\partial t^N \partial z}  \Big |+  \Big | \frac{\partial^{N+1} B(t, z, z')}{\partial t^N \partial z'}  \Big | \lesssim t^{-(N+1)} |z-z'|+|x+x'|.
\end{align}
Therefore, we obtain \begin{align*}
    &\quad \frac{\pa^NK(t, z, z')}{\pa t^N} \\
    &= \sum_{N_1+N_2=N} C_{N, N_1} \frac{\dd ^{N_1}( t^{-1/2})}{\dd t^{N_1}}\frac{\dd ^{N_2} (\sinh 2t)^{-d/2}}{\dd t^{N_2}}  e^{- B(t,z, z')}\\
    &\quad + \sum_{\substack{N_1+N_2+N_3=N\\ N_3\geq 1}} C_{N, N_1, N_2 } \frac{\dd ^{N_1}( t^{-1/2})}{\dd t^{N_1}}\frac{\dd ^{N_2} (\sinh 2t)^{-d/2}}{\dd t^{N_2}} e^{- B(t,z, z')} \frac{\partial^{N_3} B(t, z, z')}{\partial t^{N_3}}.
\end{align*}
Using the upper bound of \eqref{bound for derivatives of sinh}, \eqref{bound for derivatives of B}, and $\lambda^N e^{-\lambda} \lesssim 1$, we have
\begin{align*}
    &\quad \Big|\frac{\pa^N}{\pa t^N} K(t, z, z') \Big|\\
    &\lesssim \sum_{N_1+N_2=N}t^{-1/2-N_1-\frac{d}{2}-N_2}e^{-\frac{1}{8 t}|z-z'|^2-\frac{t}{4}|x+x'|^2} \\
    &\quad +\sum_{\substack{N_1+N_2+N_3=N\\ N_3\geq 1}} t^{-1/2-N_1-\frac{d}{2}-N_2} \big( t^{-N_3-1}|z-z'|^2+|x+x'| ^2 \big )e^{-\frac{1}{8 t}|z-z'|^2-\frac{t}{4}|x+x'|^2} \\
   &\lesssim t^{-\frac{d+1}{2}-N} e^{-\frac{1}{16t}|z-z'|^2}+\sum_{N_3\geq 1}t^{-\frac{d+1}{2}+N_3-N-1} e^{-\frac{1}{16t}|z-z'|^2}\\
  & \lesssim t^{-\frac{d+1}{2}-N} e^{-\frac{1}{16t}|z-z'|^2},
\end{align*}
which is \eqref{claim 1}. By \eqref{derivative in z}, the similar argument gives \eqref{claim 2}.

The estimates \eqref{claim 1} and \eqref{claim 2} would imply that $G_N$ is a \CZ kernel {with value in $\mathcal{H}_2$}, and hence we have $\|g_N(f)\|_{L^p(\mathbb R^{d+1})} \le C_2 \|f\|_{L^p(\mathbb R^{d+1})}$.  The reverse inequality follows from the boundedness of $g_N$,  duality argument and Lemma \ref{prop:l2gn}. In fact, by integrating \eqref{claim 1} in $t$,
\begin{align*}
  \|G_N(\cdot, z, z')\|_{\mathcal{H}_2}^2&\lesssim\int_0^1 t^{-(d+1)-2N}e^{-\frac{1}{8t}|z-z'|^2} \dd t+\int_1^\infty t^{-(d+1)-2N}e^{-\frac{1}{8t}|z-z'|^2}\dd t \\
  &\lesssim  |z'-z|^{-2(d+1)}+ e^{-|z-z'|^2}\\
  &\lesssim |z'-z|^{-2(d+1)}.
  \end{align*}
  Similarly, by \eqref{claim 2}, we have
  \begin{align*}
  \|\pa_ z G_N(\cdot, z, z')\|_{\mathcal{H}_2},   \|\pa_{z'} G_N(\cdot, z, z')\|_{\mathcal{H}_2} \lesssim   |z'-z|^{-(d+2)}.
\end{align*}
That is, $G_N(t, z, z')$ is a \CZ kernel with value in $\mathcal{H}_2$, and hence for all $p\in(1,\infty)$,
\begin{equation*}
  \|g_N(f)\|_{L^p(\R^{d+1})}=\big\|\|\tilde{g}_N(f)(\cdot, z)\|_{\mathcal{H}_2}
  \big\|_{L^p(\R^{d+1})}\leq C\|f\|_{L^p(\R^{d+1})}.\end{equation*}
 As mentioned above, by the duality argument, we can obtain the reverse inequality, and complete the proof.
\end{proof}

Next, we  define the $g_N^\ast$-function by
\begin{align}\label{equ:litpaastgndef}
  g_N^\ast(f)(z)^2  ={\int_0^\infty \!\!\int_{\R^{d+1}}
 \!\!\!\! t^{1-\frac{d+1}{2}}(1+t^{-1}|z'-z|^2)^{-N}
  |\pa_t e^{-t\AH}f(z')|^2\dd  z'\dd t}.
\end{align}
By \eqref{H decomposition}, the Schwartz kernel of the  operator $e^{-t\AH}T_m$ is
\begin{align}\label{equ:ettmkern}
  M_t(z,z')&= \sum_{\mu\in\N^d}\int_{\R}e^{i\tau(\rho-\rho')}
e^{-t(\tau^2+2|\mu|+d)}m(\tau,|\mu|)\dd \tau \Phi_{\mu}(x')\Phi_{\mu}(x)\nonumber\\
&=\sum_{k=0}^{\infty}\int_{\R}e^{i\tau(\rho-\rho')}
e^{-t(\tau^2+2k+d)}m(\tau,k)\Phi_k(x,x')\dd \tau.
\end{align}

The following result is the key estimate to show the second inequality in \eqref{ests}, that is where we use the decay assumption on $m$.

\begin{lemma}[Pointwise estimate]\label{prop:pointcon}
Under the assumption \eqref{equ:assummulf}, for any $z\in\R^{d+1}$, the following pointwise estimate
\begin{equation}
  g_{N+1}(T_mf)(z)\leq C g_N^\ast(f)(z) \label{equ:lem3.3}
\end{equation}
holds for all $0\leq N\leq \lfloor \frac{d+1}2\rfloor+1$.
\end{lemma}

\begin{proof}Observe that
\begin{align}\label{equ:patMkern}
 \pa_t^{N} \pa_s e^{-(t+s)\AH}(T_mf)(z)&= \pa_t^{N}(e^{-t\AH}T_m)(\pa_s e^{-s\AH}f)(z).
\end{align}
In particular, by choosing $s=t$, we have
\begin{align*}
  \pa_t^{N+1} e^{-2t\AH}(T_mf)(z)&= \pa_t^{N}(e^{-t\AH}T_m)(\pa_t e^{-t \AH}f)(z).
\end{align*}
In order to show \eqref{equ:lem3.3}, it suffices to show for  $t> 0$  and each $z\in\mathbb R^{d+1}$ that 
\begin{multline}\label{equ:lem1.6redu2}
    \big|\pa_t^{N+1} {e^{-2t\AH}}(T_mf)(z)\big|^2\\
   \lesssim   t^{-\frac{d+1}{2}-2N}\int_{\R^{d+1}}
(1+t^{-1}|z'-z|^2)^{-N}
  |\pa_t e^{-t\AH}f(z)|^2\dd z'.
\end{multline}

This estimate \eqref{equ:lem1.6redu2} follows from the following claim whose proof we postpone in next lemma:
\begin{equation}\label{conjecture}
  \int_{\R^{d+1}}(1+t^{-1}|z'-z|^2)^{N} |\pa_t^{N}M_t(z,z')|^2\dd z'
  \lesssim t^{-\frac{d+1}{2}-2N}.
\end{equation}
By \eqref{equ:patMkern}, \eqref{conjecture} and the Cauchy--Schwarz inequality, we obtain
\begin{align*}
 \MoveEqLeft |\pa_t^{N+1} e^{-2t\AH}(T_mf)(z)|^2 =   \big|\pa_t^{N}(e^{-t\AH}T_m)(\pa_te^{-t\AH}f)(z)|^2\\
 &=\Big|\int_{\R^{d+1}}\pa_t^{N}M_t(z,z')(\pa_te^{-t\AH}f)(z')\dd z'\Big|^2\\
 &\lesssim \int_{\R^{d+1}}(1+t^{-1}|z'-z|^2)^{N} |\pa_t^{N}M_t(z,z')|^2\dd z'\\
 &\quad \times \int_{\R^{d+1}}(1+t^{-1}|z'-z|^2)^{-N}|\pa_te^{-t\AH}f(z')|^2\dd z'\\
 &\lesssim   t^{-\frac{d+1}{2}-2N}\int_{\R^{d+1}}
(1+t^{-1}|z'-z|^2)^{-N}
|\pa_t e^{-t\AH}f(z)|^2\dd z',
\end{align*}
which gives \eqref{equ:lem1.6redu2} and hence completes the  proof.
\end{proof}

Now we turn to show the claim \eqref{conjecture}, which follows from the following lemma.

\begin{lemma}[Estimates for the kernel $M_t$]\label{lem:kernmest}
Under the assumption \eqref{equ:assummulf}, for all $0\leq N \leq \lfloor \frac{d+1}2\rfloor+1$, we have
\begin{gather}\label{equ:pointmtes}
 |\pa_t^{N}M_t(z,z')|\lesssim_N t^{-\frac{d+1}{2}-N},\\\label{equ:intestmtes}
\int_{\R^{d+1}}|z'-z|^{2N} |\pa_t^{N}M_t(z,z')|^2\dd z'
\lesssim_Nt^{-\frac{d+1}{2}-N}.
\end{gather}
\end{lemma}
\begin{proof}
We firstly prove the pointwise estimate \eqref{equ:pointmtes}. By \eqref{equ:phisumest},  \eqref{equ:ettmkern}, the Cauchy--Schwarz inequality, the $L^\infty$ bound of $m$,  and  the  fact that $\lambda^Ne^{-\lambda}\lesssim_N1$ for any $\lambda>0$, we have
\begin{align*}
 &|\pa_t^{N}M_t(z,z')|\\
 &=\Big|\sum_{\mu\in\N^d}\int_{\R}e^{i\tau(\rho-\rho')}(\tau^2+2|\mu|+d)^N
e^{-t(\tau^2+2|\mu|+d)}m(\tau,|\mu|)\dd \tau \Phi_{\mu}(x')\Phi_{\mu}(x)  \Big|\\
&\lesssim t^{-N}\sum_{\mu\in\N^d}\int_{\R}e^{-\frac{t}2(\tau^2+2|\mu|+d)}\dd \tau |\Phi_{\mu}(x')\Phi_{\mu}(x)|\\
&\lesssim  t^{-N}e^{-\frac{t}2d}\int_{\R}e^{-\frac{t}2\tau^2}\dd \tau\Big(\sum_{\mu\in\N^d} e^{-t|\mu|}\Phi_{\mu}(x')^2\Big)^\frac12
\Big(\sum_{\mu\in\N^d} e^{-t|\mu|}\Phi_{\mu}(x)^2\Big)^\frac12\\
&\lesssim  t^{-\frac{d+1}2-N},
\end{align*}
which gives \eqref{equ:pointmtes}.

Next, we  show \eqref{equ:intestmtes}. We firstly consider the case $N=0$, that is
\begin{equation}\label{equ:intesn0}
  \int_{\R^{d+1}} |M_t(z,z')|^2\dd z'
\lesssim t^{-\frac{d+1}{2}}.
\end{equation}
From \eqref{equ:ettmkern}, we know that
\begin{align*}
 M_t(z,z')&=\int_{\R}e^{-i\tau\rho'}\Big\{e^{i\tau\rho}\sum_{\mu\in\N^d}
e^{-t(\tau^2+2|\mu|+d)}m(\tau,|\mu|) \Phi_{\mu}(x')\Phi_{\mu}(x)\Big\}\dd \tau.
\end{align*}
Combining this with the Plancherel theorem in  $\rho'$, the $L^\infty$ bound of $m$,  \eqref{equ:phisumest} and  \eqref{equ:intprop}, we have
\begin{align}\nonumber
\MoveEqLeft\int_{\R^{d+1}} |M_t(z,z')|^2\dd z' = \int_{\R^{d+1}}\Big|\sum_{\mu\in\N^d}
e^{-t(\tau^2+2|\mu|+d)}m(\tau,|\mu|) \Phi_{\mu}(x')\Phi_{\mu}(x)\Big|^2\dd \tau\dd x'\\\nonumber
&\lesssim \int_{\R}e^{-2t\tau^2}\dd \tau \int_{\R^d}\sum_{\mu\in\N^d}
e^{-t(2|\mu|+d)}\Phi_{\mu}(x')^2\dd x' \sum_{\mu\in\N^d}
e^{-t(2|\mu|+d)}\Phi_{\mu}(x)^2\\
&\lesssim t^{-\frac12}(\sinh t)^{-d}e^{-td}t^{-\frac{d}{2}}\lesssim t^{-\frac{d+1}{2}}. \label{equ:mtn0est}
\end{align}
This implies \eqref{equ:intestmtes} when $N=0$.

For $N\geq1$, by the triangle inequality, we have
\begin{align}\notag
   & \int_{\R^{d+1}}|z'-z|^{2N} |\pa_t^{N}M_t(z,z')|^2\dd z'\\\nonumber
   &\lesssim \int_{\R^{d+1}} |(\rho'-\rho)^N\pa_t^{N}M_t(z,z')|^2\dd z'+ \sum_{\beta\in\N^d:\;|\beta|=N}\int_{\R^{d+1}} |(x-x')^\beta\pa_t^{N}M_t(z,z')|^2\dd z'\\
   &\eqqcolon \mathrm{I}_N+\mathrm{II}_N. \label{equ:nger1inte}
\end{align}

For the first term,  by integration by parts, we have
\begin{multline*}
(\rho-\rho')^N\pa_t^N M_t(z,z')
=\int_{\R}e^{-i\tau\rho'}\Big\{(-i)^Ne^{i\tau\rho}\\\sum_{\mu\in\N^d}\frac{\pa^N}{\pa\tau^N}\Big[(\tau^2+2|\mu|+d)^N
e^{-t(\tau^2+2|\mu|+d)}m(\tau,|\mu|)\Big]\Phi_{\mu}(x')\Phi_{\mu}(x)\Big\}\dd \tau.
\end{multline*}
In the following, we  write $ \mathbf f(\tau,k) = (\tau^2+2k+d)^N
e^{-t(\tau^2+2k+d)}m(\tau,k)$ for brevity.
By the Plancherel theorem in $\rho'$, we have
\begin{align}\notag
 \mathrm{I}_N&=\int_{\R^{d+1}} \big|(\rho'-\rho)^N\pa_t^NM_t(z,z')\big|^2\dd z' \\
 &=\int_{\R^{d+1}}\!\Big|\!\!\sum_{\mu\in\N^d}\!\frac{\pa^N}{\pa\tau^N} \mathbf f(\tau,k)\Phi_{\mu}(x')\Phi_{\mu}(x)\Big|^2 \!\dd \tau\dd x'. \label{equ:i1estN} 
\end{align}

We claim the following estimate holds under the first assumption in \eqref{equ:assummulf},
\begin{equation}\label{Key estimate 1}
  \Big|\frac{\pa^N}{\pa\tau^N}  \mathbf f(\tau,|\mu|)\Big|\lesssim t^{-\frac{N}2} e^{-\frac{t}2(\tau^2+2|\mu|+d)}.
\end{equation}
In fact, notice that for $j\ge3$, we have \begin{align*}
 \frac{\pa }{\pa \tau} (\tau^2+2|\mu|+d)=2\tau, \ \frac{\pa^2 }{\pa \tau^2} (\tau^2+2|\mu|+d)=2, \ \text{and}\   \frac{\pa^j }{\pa \tau^j} (\tau^2+2|\mu|+d)=0.
\end{align*}
Then the
\FDB gives
\begin{align*}
  \Big|\frac{\pa^{N_1}}{\pa\tau^{N_1}}(\tau^2+2|\mu|+d)^N\Big|&=\Big|\sum_{m_1+2m_2=N_1}C_{N_1, m_1,m_2}  (\tau^2+2|\mu|+d)^{N-m_1-m_2} (2\tau)^{m_1}\Big| \\
  &\lesssim  (\tau^2+2|\mu|+d)^{N-\frac{N_1}{2}}.
\end{align*}
By the \FDB again for $e^{-t (\tau^2+2|\mu|+d)}$, we have
\begin{align*}
  \Big|\frac{\pa^{N_2}}{\pa\tau^{N_2}}e^{-t(\tau^2+2|\mu|+d)}\Big|&=\Big|\sum_{n_1+2n_2=N_2}C_{N_2, n_1,n_2} e ^{-t(\tau^2+2|\mu|+d)} ( -2t\tau)^{n_1}\Big| \\
  &\lesssim  t^{N_2}e ^{-t(\tau^2+2|\mu|+d)}  (\tau^2+2|\mu|+d)^{\frac{N_2}{2}}.
\end{align*}
By the Leibniz rule and the assumption that
$$
     \Big|\frac{\pa^{N_3}}{\pa\tau^{N_3}}  m(\tau, |\mu|)\Big| \lesssim  (\tau^2+2|\mu|+d)^{-N_3/2}, 
     $$
 we obtain
\begin{align*}
    &\quad \Big|\frac{\pa^N}{\pa\tau^N} \mathbf f(\tau,|\mu|)\Big|\\
&=\bigg|\sum_{N_1+N_2+N_3=N}C_{N, N_1, N_2, N_3} \frac{\pa^{N_1}}{\pa\tau^{N_1}}(\tau^2+2|\mu|+d)^N\frac{\pa^{N_2}}{\pa\tau^{N_2}}e^{-t(\tau^2+2|\mu|+d)}\frac{\pa^{N_3}}{\pa\tau^{N_3}}  m(\tau, |\mu|)\bigg|\\
&\lesssim \sum_{N_1+N_2+N_3=N }t^{N_2}(\tau^2+2|\mu|+d)^{N-\frac{N_1}{2}+\frac{N_1}{2}-\frac{N_3}{2}}e ^{-t(\tau^2+2|\mu|+d)}.
\end{align*}
Again by using that $\lambda^Ne^{-\lambda}\lesssim_N1$ for any $\lambda>0$, we can obtain \eqref{Key estimate 1}.

Inserting \eqref{Key estimate 1} into \eqref{equ:i1estN}, we have
\begin{align}\label{equ:INest}
\mathrm{I}_N&\lesssim  t^{-N}\int_{\R^{d+1}}\Big|\sum_{\mu\in\N^d}e^{-\frac{t}2(\tau^2+2|\mu|+d)}\Phi_{\mu}(x')\Phi_{\mu}(x)\Big|^2\dd \tau\dd x'
\lesssim t^{-\frac{d+1}2-N}.
\end{align}

Next, we estimate $\mathrm{II}_N$.
It suffices to show
\begin{equation}\label{equ:iiNgoal}
  \mathrm{II}_N=\int_{\R^{d+1}} \big|(x-x')^\beta\pa_t^{N}M_t(z,z')\big|^2\dd z'\lesssim  t^{-\frac{d+1}2-N},\quad \text{for all }  |\beta|=N.
\end{equation}
We rewrite $M_t(z,z')$ as
\begin{align*}
&\pa_t^N M_t(z,z')
=(-1)^N\sum_{k=0}^{\infty}\Psi_k^N(\rho-\rho') \Phi_k(x,x'),
\end{align*}
where
$\Psi_k^N(\rho-\rho')=\int_{\R}e^{i\tau(\rho-\rho')} \mathbf f(\tau,k)\dd \tau.$
Recall that $A_j=-\frac{\pa}{\pa x_j}+x_j$. Define also $A'_j=-\frac{\pa}{\pa x_j'}+x_j'$.
From  Lemma $3.2.3$ in \cite{Than93book}, we have 
\begin{align*}
   &(x-x')^\beta \pa_t^N M_t(z,z')=\sum_{k=0}^{\infty}\sum_{\gamma,\delta}C_{\gamma,\delta}\triangle_k^{|\delta|}\Psi_k^N(\rho-\rho')(A'-A)^\gamma  \Phi_k(x,x'),
\end{align*}
where $(A'-A)^\gamma\coloneqq \prod\limits_{j=1}^d(A'_j-A_j)^{\gamma_j}$, and $\sum_{\gamma,\delta}$ denotes the sum   over all multi-indices $\gamma$ and $\delta$ satisfying
$2\delta_j-\gamma_j=\beta_j$ and $\delta_j\leq \beta_j.$
Hence,
\begin{align*}
 \MoveEqLeft (x-x')^\beta \pa_t^N M_t(z,z')\\
&=\int_{\R}e^{-i\tau\rho'}\Big\{e^{i\tau\rho}\sum_{k=0}^{\infty}\sum_{\gamma,\delta}C_{\gamma,\delta}\big[\triangle_k^{|\delta|} \mathbf f(\tau,k)\big](A'-A)^\gamma  \Phi_k(x,x')\Big\}\dd \tau.
\end{align*}
Using the Plancherel theorem in $\rho'$, we get
\begin{align}\label{equ:ii2-123}
   &\int_{\R^{d+1}} \big|(x-x')^\beta \pa_t^NM_t(z,z')\big|^2\dd z'\\\nonumber
   &=C \int_{\R^{d+1}}\bigg|\sum_{k=0}^{\infty}\sum_{\gamma,\delta}C_{\gamma,\delta}\big[\triangle_k^{|\delta|} \mathbf f(\tau,k)\big](A'-A)^\gamma  \Phi_k(x,x')\bigg|^2\dd \tau\dd x'.
\end{align}
On one hand, we claim that the second  assumption in \eqref{equ:assummulf} implies that
\begin{align}\label{Key estimate 2}
 \big|\triangle_k^{|\delta|} \mathbf f (\tau,k)\big|\lesssim t^{-(N-|\delta|)} e^{-t(\tau^2+2k+d)}.
\end{align}
Indeed, this follows from  the following Leibniz rule for finite differences,
\begin{align*}
    \MoveEqLeft \triangle_k^N(f(k) g(k) h(k))\\
    &=\sum_{m_1+m_2+m_3=N}C_{ m_1, m_2,m_3}\triangle_k^{m_1}f(k) \triangle_k^{m_2}g(k+m_1) \triangle_k^{m_3} h(k+m_1+m_2),
\end{align*}
 the second assumption in \eqref{equ:assummulf}, and the  bounds
\begin{align*}
    \big|\triangle_ k^{N_1} (\tau^2+2k+d)^N  \big| &\lesssim (\tau^2+2k+N_1+d)^{N-N_1},\\
    \big|\triangle_ k^{N_2} e^{-t(\tau^2+2k+d)}\big| &\lesssim t^{N_2}\ e^{-t(\tau^2+2k+d)}.
\end{align*}
 
On the other hand, by using  \eqref{equ:ajphiest} to expand $(A'-A)^\gamma \Phi_k(x,x')$, we obtain  
\begin{align} \label{expand A}
(A'-A)^\gamma  \Phi_k(x,x')&=\sum_{|\mu|=k}\sum_{\tau+\sigma=\gamma}A^\tau\Phi_\mu(x)(A')^\sigma\Phi_\mu(x')\nonumber\\
&=\big(2(k+1)\big)^\frac{|\gamma|}2\sum_{|\mu|=k}\sum_{\tau+\sigma=\gamma}\Phi_{\mu+\tau}(x)\Phi_{\mu+\sigma}(x').
\end{align}

Inserting \eqref{Key estimate 2} and \eqref{expand A}  into \eqref{equ:ii2-123}, we obtain \eqref{equ:iiNgoal} as follows:
\begin{align*}
   &\int_{\R^{d+1}} \big|(x-x')^\beta \pa_t^NM_t(z,z')\big|^2\dd z'\\\nonumber
   &\lesssim  \int_{\R^{d+1}}\Big|\sum_{k\ge0,\gamma,\delta }t^{-N+|\delta|} e^{-t(\tau^2+2k+d)}\big(2(k+1)\big)^\frac{|\gamma|}2\sum_{\substack{|\mu|=k \\ \tau+\sigma=\gamma }}\Phi_{\mu+\tau}(x)\Phi_{\mu+\sigma}(x')
   \Big|^2\dd \tau\dd x'\\
   &\lesssim t^{-N} \int_{\R^{d+1}}\Big|\sum_{\mu\in\N^d}e^{-\frac{t}2(\tau^2+2|\mu|+d)}\sum_{\gamma,\delta}
   \sum_{\tau+\sigma=\gamma}\Phi_{\mu+\tau}(x)\Phi_{\mu+\sigma}(x')
   \Big|^2\dd \tau\dd x'\\
   &\lesssim t^{-\frac{d+1}2-N}.
\end{align*}

Finally, by combining \eqref{equ:intesn0}, \eqref{equ:nger1inte}, \eqref{equ:INest} and \eqref{equ:iiNgoal}, we obtain \eqref{equ:intestmtes}, which concludes the proof of Lemma \ref{lem:kernmest}.\end{proof}

At last, we are ready to show the boundedness of the operator $g^*_N$ by the Hardy-Littlewood maximal function estimate and the boundedness of $g_N$.

\begin{lemma}\label{third inequality} Let $2<p<\infty$ and $N>\frac{d+1}{2}$. Then we have
\begin{align*}
    \|g_N^\ast (f)\|_{L^p(\mathbb R^{d+1})} \leq C\|f\|_{L^p(\mathbb R^{d+1})}.
\end{align*}

\end{lemma}
\begin{proof}
Let $q$ be the H\"older conjugate exponent of $p/2$. It is easy to see  that 
\begin{align*}
  \MoveEqLeft t^{-\frac{d+1}{2}}\int_{\R^{d+1}}(1+t^{-1}|z'-z|^2)^{-N}|h(z)|\dd z
  \lesssim  Mh(z'), \quad \text{for }\;  N>\frac{d+1}{2}, 
  \end{align*}
  where $M$ is  the Hardy--Littlewood maximal operator. Therefore, we have
  \begin{align*}
  &\quad \big\|g_N^\ast (f)\big\|_{L^p(\R^{d+1})}^2
  =\sup_{\substack{\|h\|_{L^q}=1}}\Big|\int_{\R^{d+1}}
  (g_N^\ast f)^2 h(z)\dd z\Big|
  \\
  &\lesssim \sup_{\substack{\|h\|_{L^q}=1}}\int_{\R^{d+1}}\!
 \int_0^\infty\!\! t|\pa_t e^{-t\AH}f(z')|^2\dd t\!\int_{\R^{d+1}}\!\!\!t^{-\frac{d+1}{2}}(1+t^{-1}|z'-z|^2)^{-N}|h(z)|\dd z \dd z'\\
 &\lesssim \sup_{\substack{\|h\|_{L^q}=1}}\int_{\R^{d+1}}
 |(g_1f)(z')|^2 M\big(\big|h(z')\big|\big)\dd z'\\
 &\lesssim \sup_{\substack{\|h\|_{L^q}=1}}\|g_1f\|^2_{L^p}\|M(|h|)\|_{L^q}\lesssim \|f\|_{L^{p}(\R^{d+1})}^2,
  \end{align*}
where we have used Lemma \ref{prop:equivnormgn} in the last inequality. This complets the proof.
\end{proof}

\subsection{Proof of Theorem \ref{thm:MikhlinmulHP}}
We are now ready to show
\begin{align*}
  \|T_m f\|_{L^p(\mathbb R^{d+1})} \lesssim \|f\|_{L^p(\mathbb R^{d+1})}, \ \text{for}\ 1<p<\infty,
\end{align*}\
under the assumption \eqref{equ:assummulf}.

 For the case $p=2$, the Plancherel theorem in $\rho$ gives
 \begin{align*}
     \|T_m f(\cdot, x)\|_{L^2_\rho}=\big\|\sum_{k=0}^\infty m(\tau, k) P_k(\mathcal F_\rho f)(\cdot, x)\big\|_{L^2_\tau}.
 \end{align*}
By the orthogonality of Hermite functions, we have
\begin{align*}
     \|T_m f(\cdot, x)\|_{L^2_\rho L^2_x}^2&=\sum_{k=0}^\infty \big\|m(\tau, k) P_k(\mathcal F_\rho f)(\cdot, x)\big\|_{L^2_\tau L^2_x}^2 \\
     &\leq \|m(\tau, k)\|_{L^\infty} \sum_{k=0}^\infty \big\|P_k(\mathcal F_\rho f)(\cdot, x)\big\|_{L^2_\tau L^2_x}^2\\
     &\lesssim \|f\|_{L^2(\mathbb R^{d+1})}^2.
\end{align*}
Hence, the result in Theorem \ref{thm:MikhlinmulHP} holds for  $p=2$.

For the case $2<p<\infty$. Let $N_0=\lfloor\frac{d+1}{2}\rfloor+1$,  by   Lemmas \ref{prop:equivnormgn},  \ref{prop:pointcon} and  \ref{third inequality}, we have
\begin{equation*}
  \|T_mf\|_{L^p}\leq C\|g_{N_0+1}(T_mf)\|_{L^p}\leq C\|g_{N_0}^\ast(f)\|_{L^p}\leq C\|f\|_{L^p}.
\end{equation*}

 Finally, the boundedness of the operator $T_m$ in $L^p(\mathbb R^{d+1})$ with $1<p<2$ follows from the duality argument. \hfill \qedsymbol



\newcommand{\doi}[1]{DOI: \href{https://doi.org/#1}{\texttt{#1}}}
\newcommand{\zbl}[1]{Zbl: \href{https://zbmath.org/?q=an\%3A#1}{\texttt{#1}}}
\newcommand{\arxiv}[1]{arXiv: \href{https://arxiv.org/abs/#1}{\texttt{#1}}}

\begin{center}

\end{center}

\end{document}